\documentclass[reqno, 12pt,a4letter]{amsart}
\usepackage{amsmath,amsxtra,amssymb,latexsym, amscd,amsthm}
\usepackage{marvosym}
\usepackage{eqname}
\usepackage{color}
\usepackage[numbers,sort&compress]{natbib}
\usepackage{booktabs}
\usepackage[utf8]{inputenc}
\usepackage[mathscr]{euscript}
\usepackage{mathrsfs}
\usepackage[english]{babel}

\newtheorem {theorem}{Theorem}[section]

\newtheorem {lemma}{Lemma}[section]

\newtheorem {proposition}{Proposition}[section]
\newtheorem {example}{Example}[section]
\newtheorem {definition}{Definition}[section]
\newtheorem {remark}{Remark}[section]

\def\ar{a\kern-.370em\raise.16ex\hbox{\char95\kern-0.53ex\char'47}\kern.05em}
\def\ees{{\accent"5E e}\kern-.385em\raise.2ex\hbox{\char'23}\kern-.08em}
\def\eex{{\accent"5E e}\kern-.470em\raise.3ex\hbox{\char'176}}
\def\AR{A\kern-.46em\raise.80ex\hbox{\char95\kern-0.53ex\char'47}\kern.13em}
\def\EES{{\accent"5E E}\kern-.5em\raise.8ex\hbox{\char'23 }}
\def\EEX{{\accent"5E E}\kern-.60em\raise.9ex\hbox{\char'176}\kern.1em}
\def\ow{o\kern-.42em\raise.82ex\hbox{
  \vrule width .12em height .0ex depth .075ex \kern-0.16em \char'56}\kern-.07em}
\def\OW{O\kern-.460em\raise1.36ex\hbox{
\vrule width .13em height .0ex depth .075ex \kern-0.16em \char'56}\kern-.07em}
\def\UW{U\kern-.42em\raise1.36ex\hbox{
\vrule width .13em height .0ex depth .075ex \kern-0.16em \char'56}\kern-.07em}
\def\DD{D\kern-.7em\raise0.4ex\hbox{\char '55}\kern.33em}
\newtheorem{algorithm}{Algorithm}[section]
\newcommand{\argmin}{\operatornamewithlimits{argmin}}

\title{Polynomial mathematical programs with equilibrium constraints and semidefinite programming relaxations$^*$}

\author{Liguo Jiao}
\address[Liguo Jiao]{School of Mathematical Sciences, Soochow University, Suzhou 215006, Jiangsu Province, People's Republic of China}
\email{hanchezi@163.com}

\author{Jae Hyoung Lee$^\dag$}
\address[Jae Hyoung Lee]{Department of Applied Mathematics, Pukyong National University, Busan 48513, Republic of Korea}
\email{mc7558@naver.com}

\author{Ti\ees n-S\ow N Ph\d{a}m$^\ddag$}
\address[Ti\ees n-S\ow n Ph\d{a}m]{Department of Mathematics, University of Dalat, 1 Phu Dong Thien Vuong, Dalat, Vietnam}
\email{sonpt@dlu.edu.vn}

\thanks{$^*$This work was supported by the National Research Foundation of Korea (NRF) Grant funded by the Korea government (MSIP) (NRF-2018R1C1B6001842).}

\thanks{$^\dag$Corresponding author}

\thanks{$^\ddag$This author is partially supported by Vietnam National Foundation for Science and Technology Development (NAFOSTED)}


\date{ \today}

\begin{document}
\maketitle

\begin{abstract}
This paper focuses on the study of a mathematical program with equilibrium constraints, where the objective and the constraint functions are all polynomials.
We present a method for finding its global minimizers and global minimum using a hierarchy of semidefinite programming (SDP) relaxations
and prove the convergence result for the method. Numerical experiments are presented to show the efficiency of the proposed algorithm.
\end{abstract}

\section{Introduction}
Mathematical Programs with Equilibrium Constraints (MPECs for short) form an important class of (nonlinear) constrained optimization problems, in which the variables satisfy a finite number of constraints together with an equilibrium condition such as variational inequalities or complementarity conditions. The term ``MPEC" is believed to have been put in \cite{Harker1988}, and
the word ``equilibrium'' is used since the variational inequality constraints of the MPEC typically model specific equilibrium phenomena that occur in engineering and economic applications.
MPECs are natural extensions of optimization problems with variational inequality constraints~\cite{Facchinei2003,Lucet2002,Ye2010}, bilevel optimization problems~\cite{Dempe2014, Jeyakumar2016, Lin2014, Nie2017}, semi-infinite optimization problems~\cite{Jongen2011,Lasserre2012,Stein2001}, minimax (robust) optimization problems~\cite{Beck2009,Bertsekas2010,Gaudioso2006}. There is a large literature on all aspects of MPECs; we refer the reader to the comprehensive monographs \cite{Facchinei2003, Luo1996, Outrata1998} with the references therein.

MPECs have been studied in the past years, with several solution methods developed; for example, the elastic mode approach~\cite{Anitescu2005,Anitescu2007}, relaxation schemes~\cite{DeMiguel2005,Steffensen2010}, smoothing method~\cite{Facchinei1999}, sequential quadratic programming methods~\cite{Fletcher2006}, interior-point method~\cite{Liu2004, Luo1996}, exact penalization approach~\cite{Jiang2018,Scholtes1999}, see also the references therein.

We would like to note that, in general, an MPEC is a {\em nonconvex and nondifferentiable} optimization problem that includes certain combinatorial features in its constraints. Therefore, it is computationally very difficult to solve, especially if we wish to find a {\em globally} optimal solution.

In this paper, we are interested in an MPEC with {\em polynomial data} which admits the following mathematical form:
\begin{eqnarray}
f_* := &\min\limits_{(x,y)\in \mathbb{R}^n\times\mathbb{R}^m}& f(x,y) \notag\\
&{\rm subject\ to}& (x,y)\in A \cap B, \eqname{MPEC}\\
&& \varphi(x,y,v)\ge 0, \ \forall v\in B(x), \notag
\end{eqnarray}
where $f\colon \mathbb{R}^n \times \mathbb{R}^m\to\mathbb{R},$ $\varphi\colon \mathbb{R}^n\times \mathbb{R}^m\times \mathbb{R}^m\to \mathbb{R}$ are polynomial functions, and we call $\varphi(x,y,v)\ge 0$ for all $v\in B(x)$ the {\it equilibrium constraints}$,$ $A, B\subset \mathbb{R}^n \times \mathbb{R}^m$ are basic closed semi-algebraic sets, and $B(x) := \{y \in \mathbb{R}^m : (x, y) \in B\}$ for $x \in \mathbb{R}^n.$

We aim to propose a computational method for finding/approximating the optimal value $f_*$ (and minimizers, if possible) of Problem~{\rm (MPEC)}. To do this, we first define the optimal value function for the equilibrium constraints by $J(x, y) := \min_{v\in B(x)}\varphi(x,y,v).$
Under a blanket assumption, that is commonly used in polynomial optimization (see \cite{Jeyakumar2016, Lasserre2010, Lasserre2011, Lasserre2012}), we show that the function $J$ is well-defined on some compact set $\Omega \subset \mathbb{R}^n \times \mathbb{R}^m$ containing the set $B$ and that (MPEC)  is equivalent to the following constrained optimization problem
\begin{eqnarray}
\mathrm{val}({\rm P}) := &\min\limits_{(x,y)\in \mathbb{R}^n\times\mathbb{R}^m}& f(x,y) \notag\\
&{\rm s.t.}& (x,y)\in A\cap B, \eqname{P}\\
&& J(x,y)\ge 0. \notag
\end{eqnarray}
Note that, in general, the function $J$ is not polynomial and so Problem (P) is not a polynomial optimization problem. Nevertheless, by using the ``joint+marginal'' approach for parametric polynomial optimization developed by Lasserre in \cite{Lasserre2010}, we can construct a sequence of polynomials $J_k \colon \mathbb{R}^n \times \mathbb{R}^m \rightarrow \mathbb{R}$ (with degree at most $2k, k \in \mathbb{N})$ that approximate from below the function $J$ on $\Omega,$ and with the strong property that $J_k \to J$ in the $L_1$-norm as $k \to \infty.$ (In particular, $J_{k_\ell} \to J$ almost uniformly on $\Omega$ for some subsequence $k_{\ell}, \ell \in \mathbb{N}.$) Then, ideally, we could solve the nested sequence of polynomial optimization problems:
\begin{eqnarray}
&\min\limits_{(x, y)\in \mathbb{R}^n\times\mathbb{R}^m}& f(x,y)\notag \\
&{\rm s.t.}& (x, y)\in A \cap B, \eqname{P$^k$}\\
&& J_k(x, y) \ge 0.\notag
\end{eqnarray}
For fixed $k,$ we may approximate (and often solve exactly) Problem (P$^k$) by solving a hierarchy of semidefinite programming relaxations, as defined by Lasserre in \cite{Lasserre2001}. However, as the feasible set of Problem (P$^k$) may be empty, we relax the constraint $J_k(x, y) \ge  0$ to $J_k (x, y) \ge - \epsilon$ for some scalar $\epsilon > 0,$ which can adjust dynamically during the algorithm. Moreover, in order to establish convergence of our algorithm, we also relax the sets $A$ and $B$ to basic closed semi-algebraic sets $A_\epsilon$ and $B_\epsilon,$ respectively. In other words, we will relax the problem (P$^k$) to the following polynomial optimization problem
\begin{eqnarray}
\mathrm{val(P_\epsilon^k)} := &\min\limits_{(x, y)\in \mathbb{R}^n\times\mathbb{R}^m}& f(x,y)\notag \\
&{\rm s.t.}& (x, y)\in A_\epsilon \cap B_\epsilon, \eqname{P$^k_\epsilon$}\\
&& J_k(x, y) \ge -\epsilon.\notag
\end{eqnarray}
As mentioned above, this problem can be solved by using Lasserre-type semidefinite programming relaxations.
Finally, let $v_{\epsilon}^k := \min_{1 \le i \le k} \mathrm{val}({\rm P}_{\epsilon}^i).$ Then we can show that the following statements are valid:
\begin{itemize}
\item[{\rm (i)}]  The sequence $\{v_{\epsilon}^k\}_{k \in \mathbb{N}}$ is bounded from above by $f_* + \epsilon.$
\item[{\rm (ii)}]  For all $\epsilon \in (0, +\infty)$ except finitely many points, the sequence $\{v_{\epsilon}^k\}_{k \in \mathbb{N}}$ converges  monotonically, decreasing to the optimal value $\mathrm{val(P_\epsilon)}$ of the problem
\begin{eqnarray}
\mathrm{val(P_\epsilon)} := &\min\limits_{(x,y)\in \mathbb{R}^n\times\mathbb{R}^m}& f(x,y)\notag \\
&{\rm s.t.}& (x,y)\in A_\epsilon\cap B_\epsilon, \eqname{P$_\epsilon$}\\
&& J(x,y)\ge-\epsilon.\notag
\end{eqnarray}
\item[{\rm (iii)}] The following relation holds true
\begin{eqnarray*}
\mathrm{val(P_\epsilon)} &=& f_* + c \epsilon^q + o(\epsilon^q) \quad \textrm{ as } \quad \epsilon \to 0^+
\end{eqnarray*}
for some $c \le 0$ and $q > 0.$ In particular,
\begin{eqnarray*}
\lim_{\epsilon \to 0^+} \lim_{k \to \infty} v_{\epsilon}^k &=& f_*.
\end{eqnarray*}
\end{itemize}

Our approach extends the sequential SDP relaxations, introduced in \cite{Jeyakumar2016} for bilevel polynomial optimization problems and in \cite{Lasserre2012} for semi-infinite optimization problems, to mathematical programs with equilibrium constraints.
It is worth emphasizing that it is different from the papers \cite{Jeyakumar2016,Lasserre2012} that we do not make technical assumptions concerning the interior of the feasible sets; moreover, the idea in this paper can be used to modify the algorithms in the two papers cited, so that we can drop these technical assumptions.

The rest of this paper is organized as follows. Section~\ref{sect:2} contains some preliminaries on semi-algebraic geometry.
Section~\ref{sect:3} presents convergence of our sequential SDP relaxation scheme for solving Problem (MPEC). Finally, conclusions are given in Section~\ref{sect:4}.

\section{Preliminaries}\label{sect:2}

Throughout this paper, $\mathbb{R}^n$ denotes the Euclidean space with dimension $n.$ The open ball in $\mathbb{R}^n$ centered at $x$ with radius $\rho$ is denoted by $\mathbb{B}(x,\rho).$ We also use $\mathbb{N}$ to denote all the nonnegative integers. For a set $D$ in $\mathbb{R}^n,$ we use ${\rm cl}(D)$ and ${\rm int}(D)$ to denote the closure and interior of $D,$ respectively.

Denote by $\mathbb{R}[x]$ the ring of polynomials in $x:=(x_1, \ldots, x_n)$ with real coefficients.
For a polynomial $f,$ we use $\deg f$ to denote the degree of $f.$
We say that a polynomial $f\in\mathbb{R}[x]$ is {\it sum of squares} if there exist polynomials $q_l,$ $l = 1,\ldots,r,$ such that $f =\sum_{l=1}^{r}q_{l}^2.$
The set consisting of all sum of squares polynomial in $x$ is denoted by $\mathrm{\Sigma}^2[x].$
Moreover, the set of all sum of squares polynomials in $x$ with degree at most $d$ is denoted by $\mathrm{\Sigma}^2_d[x].$
For a multi-index $\alpha := (\alpha_1, \ldots, \alpha_n) \in \mathbb{N}^n,$ let $|\alpha|:=\sum_{i=1}^n\alpha_i$ and
the notation $x^\alpha$ stands for the monomial $x_1^{\alpha_1}\cdots x_n^{\alpha_n}.$ Finally, let $\mathbb{N}^n_d := \{\alpha \in \mathbb{N}^n :  |\alpha|\le d\}.$

The following Archimedean property is commonly used in polynomial optimization (see \cite{Ha2017, Lasserre2009} and references therein).
\begin{definition}{\rm
A finite system $\{f_1, \ldots, f_s\} \subset \mathbb{R}[x]$ is said to satisfy the {\it Archimedean property} if there exists $R > 0$ such that the quadratic polynomial $x\mapsto R - \|x\|^2$ can be written in the form
\begin{eqnarray*}
R - \|x\|^2 &=& \sigma_0+\sum_{j = 1}^{s} \sigma_j f_j(x),
\end{eqnarray*}
for some sum of squares polynomials $\sigma_j \in \mathrm{\Sigma}^2[x].$
}\end{definition}

In what follows, we recall some notions and results of semi-algebraic geometry, which can be found in \cite{Bochnak1998,Ha2017}.

\begin{definition}{\rm
A subset of $\mathbb{R}^n$ is called {\em semi-algebraic} if it is a finite union of sets of the form $\{x\in\mathbb{R}^n:f_i(x)=0, \ i=1,\ldots,k, \ f_i(x)>0, \ i=k+1,\ldots, p\},$ where all $f_i$ are polynomials. If $A\subset \mathbb{R}^n$ and $B\subset\mathbb{R}^p$ are semi-algebraic sets, then the map $F\colon A\to B$ is said to be {\em semi-algebraic} if its graph $\{(x,y)\in A\times B : y=F(x)\}$ is a semi-algebraic subset in $\mathbb{R}^n\times\mathbb{R}^p.$}
\end{definition}

Note that semi-algebraic sets and functions enjoy a number of remarkable properties. We summarize some of the important properties which will be used in this paper.

\begin{proposition}\label{property of SA}
The following statements hold$:$
\begin{itemize}
\item[{\rm (i)}] Each semi-algebraic set in $\mathbb{R}$ is a finite union of intervals and points.
\item[{\rm (ii)}] Finite union $($resp.$,$ intersection$)$ of semi-algebraic sets is semi-algebraic.
\item[{\rm (iii)}]The Cartesian product $($resp.$,$ complement$,$ closure$)$ of semi-algebraic sets is semi-algebraic.
\item[{\rm (iv)}] If $f,$ $g$ are semi-algebraic functions on $\mathbb{R}^n$ and $\lambda\in\mathbb{R},$ then $f+g,$ $fg$ and $\lambda f$ are all semi-algebraic functions.
\item[{\rm (v)}] If $f$ is a semi-algebraic function on $\mathbb{R}^n$ and $\lambda\in\mathbb{R},$ then $\{x\in \mathbb{R}^n : f(x)\le\lambda\},$ $\{x\in \mathbb{R}^n : f(x)<\lambda\}$ and $\{x\in \mathbb{R}^n : f(x)=\lambda\}$ are all semi-algebraic sets.
\item[{\rm (vi)}] A composition of semi-algebraic maps is a semi-algebraic map.
\end{itemize}
\end{proposition}

\begin{theorem}[Tarski--Seidenberg Theorem]\label{Tarski Seidenberg Theorem}
The image and inverse image of a semi-algebraic set under a semi-algebraic map are semi-algebraic sets.
In particular$,$ the projection of a semi-algebraic set is still a semi-algebraic set.
\end{theorem}

\begin{remark}\label{Remark1}{\rm
If $A\subset \mathbb{R}^n,$ $B\subset\mathbb{R}^m,$ and $C \subset\mathbb{R}^n\times\mathbb{R}^m$ are semi-algebraic sets, then we see that
$U:=\{x\in A:(x,y)\in C, \ \forall y\in B\}$ is also a semi-algebraic set.
To see this, from Theorem~\ref{Tarski Seidenberg Theorem}, we see that $\{x\in A: \exists y\in B \textrm{ s.t. } (x,y)\in C\}$ is semi-algebraic.
As the complement of $U$ is the union of the complement of $A$ and the set $\{x\in A :\exists y\in B \textrm{ s.t. } (x,y)\notin C\},$
it follows that the complement of $U$ is semi-algebraic by Proposition~\ref{property of SA}(iii).
Thus, $U$ is also semi-algebraic by Proposition~\ref{property of SA}(iii).
In general, if we have a finite collection of semi-algebraic sets, then any set obtained from them by a finite chain of quantifiers is also semi-algebraic.
}\end{remark}

\begin{lemma}[Monotonicity Lemma]\label{Monotonicity Theorem}
Let $\phi  \colon (a,b)\to\mathbb{R}$ be a semi-algebraic function. Then there exist $a=a_0<a_1<\cdots<a_s<a_{s+1}=b$ such that$,$ for each $i=0,1,\ldots,s,$ the restriction $\phi _{|_{(a_i,a_{i+1})}}$ is analytic$,$ and either constant$,$ or strictly increasing or strictly decreasing.
\end{lemma}

\begin{lemma}[Growth Dichotomy Lemma]\label{Growth Dichotomy Lemma}
Let $\phi \colon  (0, \epsilon)\to \mathbb{R}$ be a semi-algebraic function with $\phi(t) \neq 0$ for all $t\in (0, \epsilon).$
Then there exist constants $c \neq 0$ and $q \in \mathbb{Q}$ such that $\phi(t) = ct^q + o(t^q )$ as $t \to 0^+.$
\end{lemma}

\section{The main result}\label{sect:3}

Hereafter that$,$ we assume that $f, g_i, h_j \in \mathbb{R}[x, y], i = 1, \ldots, r, j = 1, \ldots, s,$ and put
\begin{eqnarray*}
A &:=& \{(x, y) \in \mathbb{R}^n \times \mathbb{R}^m : g_i(x, y) \ge 0, \ i = 1,\ldots, r\}, \\
B &:=& \{(x, y) \in \mathbb{R}^n \times \mathbb{R}^m : h_j(x, y) \ge 0, \ j = 1,\ldots, s\}.
\end{eqnarray*}
Moreover, for each $x \in \mathbb{R}^n,$ we define
$$B(x) := \{y \in \mathbb{R}^m : h_j(x, y) \ge 0, \ j = 1,\ldots, s\}.$$
Now, we recall Problem $({\rm MPEC})$, which admits the following form:
\begin{eqnarray}
f_\ast:=
&\min\limits_{(x,y)\in \mathbb{R}^n\times\mathbb{R}^m}& f(x,y) \notag\\
&{\rm s.t.}& (x,y)\in A, \ y\in B(x), \eqname{MPEC}\\
&& \varphi(x,y,v)\ge 0, \ \forall v\in B(x),\notag
\end{eqnarray}
where $\varphi\colon \mathbb{R}^n\times \mathbb{R}^m\times \mathbb{R}^m\to\mathbb{R}$ is a polynomial function, and we call $\varphi(x,y,v)\ge 0$ for all $v\in B(x)$ the {\it equilibrium constraints}. Throughout this work we shall assume that the feasible set of Problem ${\rm (MPEC)}$ is nonempty. We also need the following blanket assumption:
\begin{enumerate}
\item[{(H1)}] {\em There exists a compact semi-algebraic set $\mathrm{\Omega}\subset\mathbb{R}^n\times\mathbb{R}^m$ such that $B\subset \mathrm{\Omega}$ and
for each $x\in{\mathrm{Proj}}_x\mathrm{\Omega},$ the set $B(x)$ is nonempty.}
\end{enumerate}

Here and in  the following, ${\mathrm{Proj}}_x\mathrm{\Omega}$ stands for the image of $\mathrm{\Omega}$ via the canonical projection $\mathbb{R}^n\times\mathbb{R}^m\to\mathbb{R}^n,$ $(x,y)\mapsto x.$

Now, we define the function $J\colon \mathrm{\Omega}\to \mathbb{R}$ by
\begin{equation*}
(x,y)\mapsto J(x,y):=\min_{v\in B(x)}\varphi(x,y,v).
\end{equation*}

\begin{lemma}\label{Lemma1}
Suppose that the assumption ${\rm (H1)}$ satisfies. Then the function $J \colon \mathrm{\Omega}\to\mathbb{R}$ is well-defined and semi-algebraic$;$ furthermore$,$ it is lower semicontinuous$,$ i.e.$,$ for any $\bar x \in \mathrm{\Omega}$ we have
$$\liminf_{x\to \bar x,x\in\mathrm{\Omega}}\phi(x)\ge  \phi(\bar x).$$
\end{lemma}

\begin{proof}
It follows from the assumption ${\rm (H1)}$ that $B(x)$ is a nonempty and compact set.
Moreover, since the polynomial function $\varphi$ is continuous, the optimal value function $(x,y)\mapsto J(x,y)$ is well-defined.
Also, by Tarski--Seidenberg Theorem (see Theorem~\ref{Tarski Seidenberg Theorem}), it is not hard to see that $J$ is semi-algebraic.
To prove the lower semi-continuity of $J$, let $(x^k,y^k) \in \Omega$ be a sequence such that $(x^k,y^k)\to(\bar x,\bar y)$ as $k\to\infty.$
Since $B(x^k)\subset \mathrm{Proj}_xB$ and $\mathrm{Proj}_xB$ is a compact set (as $B$ is compact), $B(x^k)$ is compact.
Also, since $\varphi$ is continuous on $\mathbb{R}^n\times\mathbb{R}^m\times\mathbb{R}^m,$ there exists $v^k\in B(x^k)$ such that $J(x^k,y^k)=\varphi(x^k,y^k,v^k).$
Moreover, as $v^k\in B(x^k)$, without loss of generality, we assume that $v^k\to \bar v$ as $k\to \infty.$
Note that $h_j(x^k,v^k)\ge0$ for $j=1,\ldots,s,$ and all $k.$ This implies that $h_j(\bar x,\bar v)\ge0$ for $j=1,\ldots,s,$ and so, $\bar v\in B(\bar x).$ Finally, we have
\begin{eqnarray*}
\liminf_{k\to\infty}J(x^k,y^k)
=\lim_{k\to\infty}\varphi(x^k,y^k, v^k)
=\varphi(\bar x,\bar y,\bar v)\ge \min_{v\in B(\bar x)} \varphi(\bar x,\bar y,v)=J(\bar x,\bar y),
\end{eqnarray*}
and thus, $J$ is a lower semicontinuous function.
\end{proof}

Now, along with Lemma~\ref{Lemma1}, and under the assumption ${\rm (H1)},$  Problem ${\rm (MPEC)}$ can be rewritten as
\begin{eqnarray}
&\min\limits_{(x,y)\in \mathbb{R}^n\times\mathbb{R}^m}& f(x,y) \notag\\
&{\rm s.t.}& (x,y)\in A\cap B, \eqname{P}\\
&& J(x,y)\ge 0. \notag
\end{eqnarray}
Moreover, if the function $J$ is continuous, then Problem ${\rm (P)}$ has at least a minimizer. Unfortunately, the following example shows that the function $J$ may not be continuous.
\begin{example}{\rm
In $\mathbb{R}^2$, let
\begin{eqnarray*}
A&:=&\{(x,y)\in\mathbb{R}^2 : 1-x^2\ge0, \ 1-y^2\ge0\},\\
B&:=&\{(x,y)\in\mathbb{R}^2 : 1-y^2\ge0, \ -xy\ge0\}.
\end{eqnarray*}
Consider a mathematical programming with equilibrium constraint as follows:
\begin{eqnarray*}
&\min\limits_{(x,y)\in \mathbb{R}\times\mathbb{R}}& f(x,y) \\
&{\rm s.t.}& (x,y)\in A, \ y\in B(x):=\{v\in\mathbb{R}:1-v^2\ge0, \ -xv\ge0\}, \\
&& \varphi(x, y, v) := v - y\ge 0, \ \forall v\in B(x).
\end{eqnarray*}
It is not hard to check that for all $(x,y)\in\mathbb{R} \times\mathbb{R},$
\begin{equation*}
J(x,y) = \min_{v\in B(x)}(v-y) =
\begin{cases}
-y & \textrm{if } x<0, \\
-1-y & \textrm{if } x\ge0.
\end{cases}
\end{equation*}
Therefore, $J$ is not continuous at all points $(x,y)\in\{0\}\times\mathbb{R}.$
Note that in this example,
\begin{equation*}
B(x) = \{v\in\mathbb{R}:1-v^2\ge0, \ -xv\ge0\} =
\begin{cases}
\left[-1,0\right] & \textrm{if } x>0, \\
\left[0,1\right] & \textrm{if } x<0, \\
\left[-1,1\right] & \textrm{if } x=0.
\end{cases}
\end{equation*}
Hence the set-valued map $B$ is not lower semicontinuous\footnote{We say that the set-valued map $B\colon \mathrm{Proj}_x \mathrm{\Omega} \rightrightarrows \mathbb{R}^m,$ $x \rightrightarrows B(x),$ is said to be {\em lower semi-continuous} at $\bar x \in \mathrm{Proj}_x \mathrm{\Omega}$ iff for each open set $V \subset \mathbb{R}^m$ satisfying $B(\bar x) \cap V \neq \emptyset,$ there exists $\epsilon > 0$ such that $B(x) \cap V \neq \emptyset$ whenever $\|x-\bar x\| < \epsilon.$} at $x = 0.$
}\end{example}

The following lemma provides a sufficient condition for $J$ being continuous.

\begin{lemma}\label{Lemma2}
Suppose that the assumption ${\rm (H1)}$ satisfies. If the set-valued mapping
$$B\colon \mathrm{Proj}_x \mathrm{\Omega} \rightrightarrows \mathbb{R}^m, \quad x \rightrightarrows B(x),$$
is lower semi-continuous$,$ then the function $J$ is continuous.
\end{lemma}

\begin{proof}
It is easy to see that $J$ is continuous in the variable $y.$ So it remains to show that $J$ is continuous in the variable $x.$
To do this, take any $(\bar{x}, \bar{y}) \in \mathrm{\Omega}$ and let $(x^k,\bar{y})\in\mathrm{\Omega}$ be a sequence such that $\lim_{k \to \infty} x^k = \bar{x}.$
We shall show that
\begin{equation*}
\lim_{k \to \infty} J(x^k, \bar{y}) = J(\bar{x}, \bar{y}).
\end{equation*}
In fact, since $J$ is lower semi-continuous, we have
\begin{equation*}
J(\bar{x}, \bar{y}) \le \liminf_{k \to \infty} J(x^k, \bar{y}).
\end{equation*}
Hence, it suffices to show that
\begin{equation*}
\limsup_{k \to \infty} J(x^k, \bar{y}) \le J(\bar{x}, \bar{y}).
\end{equation*}
To see this, let $\bar{v} \in B(\bar{x})$ be such that $J(\bar{x}, \bar{y}) = \varphi(\bar{x}, \bar y, \bar{v}).$
Take any $\epsilon > 0.$ Then
\begin{equation*}
\mathbb{B}(\bar{v}, \epsilon) \cap B(\bar{x}) \ne \emptyset.
\end{equation*}
Since $B$ is lower semi-continuous, it follows that
\begin{equation*}
\mathbb{B}(\bar{v}, \epsilon) \cap B(x^k) \ne \emptyset \quad \textrm{ for all } \quad k \gg 1.
\end{equation*}
In particular, for each $k \gg 1,$ there exists $v^k \in \mathbb{B}(\bar{v}, \epsilon) \cap B(x^k).$ By the definition of $J,$ then
\begin{equation*}
J(x^k, \bar{y}) = \min_{v \in B(x^k)} \varphi(x^k,\bar y, v) \le \varphi(x^k,\bar y, v^k).
\end{equation*}
Therefore,
\begin{equation*}
\limsup_{k \to \infty} J(x^k, \bar{y}) \le \limsup_{k \to \infty} \varphi(x^k,\bar y, v^k) = \varphi(\bar{x}, \bar y,\bar{v}_\epsilon)
\end{equation*}
for some $\bar{v}_\epsilon \in \mathbb{B}(\bar{v}, \epsilon).$ Letting $\epsilon \to 0,$ the desired result follows.
\end{proof}

\begin{remark}{\rm
By Lemma~\ref{Lemma2}, if the set-valued map $B$ does not depend on the variable $x$ then the function $J$ is continuous.
}\end{remark}

\subsection{The $\epsilon$-approximation of Problem (MPEC)}

Let $\epsilon\ge0,$ and consider the following perturbed sets of $A$ and $B$:
\begin{eqnarray*}
A_{\epsilon} &:=&\{(x,y)\in\mathrm{\Omega} : g_i(x,y)\ge-\epsilon, \ i=1,\ldots,r\},\\
B_{\epsilon} &:=&\{(x,y)\in\mathrm{\Omega} : h_j(x,y)\ge-\epsilon, \ j=1,\ldots,s\}.
\end{eqnarray*}
Then $A_\epsilon$ and $B_\epsilon$ are nonempty compact semi-algebraic sets. We would like to mention that it is different from \cite{Jeyakumar2016,Lasserre2012} that we perturbed the sets $A$ and $B.$ It turns out that we do not make the assumptions concerning the interior of $A$ and $B$ to obtain the results in this paper, as we shall see.

Next, we define an $\epsilon$-approximation of Problem ${\rm (P)}$ as follows:
\begin{eqnarray}
&\min\limits_{(x,y)\in \mathbb{R}^n\times\mathbb{R}^m}& f(x,y)\notag \\
&{\rm s.t.}& (x,y)\in A_\epsilon\cap B_\epsilon, \eqname{P$_\epsilon$}\\
&& J(x,y)\ge-\epsilon.\notag
\end{eqnarray}
We denote the optimal value of problems ${\rm (P)}$ and ${\rm (P_\epsilon)}$ by $\mathrm{val}{\rm (P)}$ and $\mathrm{val}{\rm (P_\epsilon)},$ respectively.

\begin{lemma}\label{Lemma3}
Suppose that the assumption ${\rm (H1)}$ satisfies. Then the following two statements hold$:$
\begin{itemize}
\item[{\rm (i)}] The function
\begin{equation*}
[0,+\infty)\to \mathbb{R}, \ \ \epsilon \mapsto \mathrm{val}{\rm (P_\epsilon)},
\end{equation*}
is well-defined$,$ non-increasing and semi-algebraic. In particular$,$ it is analytic except at finitely many points.

\item[{\rm (ii)}] If $J$ is continuous$,$ then a global minimizer for Problem ${\rm (P_\epsilon)}$ exists for all $\epsilon\ge0.$
Furthermore$,$ there exist $\bar\epsilon>0,$ $q\in\mathbb{Q}$ with $q>0,$ and $c\le0$ such that for all $\epsilon\in[0,\bar\epsilon],$
\begin{equation}\label{Lemma3rel1}
\mathrm{val}{\rm (P_\epsilon)}=\mathrm{val}{\rm (P)}+c\epsilon^q+o(\epsilon^q).
\end{equation}
\end{itemize}
\end{lemma}

\begin{proof}
(i) By the assumptions, the feasible set of Problem ${\rm (P)}$ (and hence of $({\rm P}_\epsilon)$) is nonempty and bounded.
In particular, $\mathrm{val}{\rm (P_\epsilon)}$ is finite for all $\epsilon\ge0,$ and so, the function $\epsilon \mapsto \mathrm{val}{\rm (P_\epsilon)}$ is well-defined.

From the definition of Problem~$({\rm P}_{\epsilon})$, it is clear that if $0 \le \epsilon_1 \le \epsilon_2$, then ${\rm val}{\rm (P_{\epsilon_1})} \ge {\rm val}{\rm(P_{\epsilon_2})}$.

By Lemma \ref{Lemma1}, the function $J$ is well-defined and semi-algebraic. Let
\begin{eqnarray*}
X & := & \{(\epsilon, x, y) \in [0, +\infty) \times \mathrm{\Omega}  :  g_i(x, y) \ge-\epsilon, i = 1, \ldots, r, \\
&  & \hspace{3.75cm} h_j(x, y) \ge-\epsilon, j = 1, \ldots, s,  J(x, y) \ge - \epsilon \}, \\
Y & := & \{(\epsilon, x, y) \in X  : f(x, y) \le f(a, b), \  \forall (\epsilon, a, b) \in X\}.
\end{eqnarray*}
We can verify that $X$ and $Y$ are semi-algebraic sets by Proposition~\ref{property of SA}(iii)-(iv) and Remark~\ref{Remark1}.
Further, by Tarski--Seidenberg Theorem (see Theorem~\ref{Tarski Seidenberg Theorem}), the function
$$[0,+\infty) \rightarrow \mathbb{R}, \quad \epsilon \mapsto {\rm val}({\rm P}_{\epsilon}),$$
 is semi-algebraic, and so it is analytic except at finitely many points (due to Lemma~\ref{Monotonicity Theorem}).

(ii) Assume that $J$ is continuous. Then for each $\epsilon \ge 0,$ the constraint set of Problem~$({\rm P}_{\epsilon})$ is nonempty compact, and so a global minimizer for $({\rm P}_{\epsilon})$ exists (because the objective polynomial $f$ is continuous).

We now claim that the function $\epsilon \mapsto \mathrm{val}{\rm (P_\epsilon)}$ is right continuous at $0.$
To see this, let $\epsilon^k>0,$ $k=1,2,\ldots,$ with $\epsilon^k\downarrow0,$ and let each $(x^k,y^k)$ be an optimal solution of $({\rm P}_{\epsilon^k}).$
Since $\mathrm{\Omega}$ is compact, without loss of generality, we assume that $(x^k,y^k)\to(x^\ast,y^\ast)$ as $k\to\infty.$
Note that the functions $g_i,$ $h_j,$ and $J$ are continuous.
So, we can easily verify that $(x^\ast,y^\ast)$ is a feasible solution of ${\rm (P)}.$
This yields that
\begin{equation*}
\mathrm{val}{\rm (P)}\le f(x^\ast,y^\ast)=\lim_{k\to\infty}f(x^k,y^k)=\lim_{k\to\infty}{\rm val}({\rm P}_{\epsilon^k})=\lim_{\epsilon\downarrow0^+}{\rm val}({\rm P}_{\epsilon})\le\mathrm{val}{\rm (P)},
\end{equation*}
where the last inequality follows from the nonincreasing property of the function $\epsilon \mapsto \mathrm{val}{\rm (P_\epsilon)},$
and so, the function $\epsilon \mapsto \mathrm{val}{\rm (P_\epsilon)}$ is right continuous at $0.$

Define the function
\begin{equation*}
\phi\colon [0,+\infty)\to \mathbb{R}, \ \ \epsilon\mapsto\mathrm{val}{\rm (P_\epsilon)}-\mathrm{val}{\rm (P)}.
\end{equation*}
Then $\phi$ is a nonincreasing semi-algebraic function and
\begin{equation}\label{Lemma3rel2}
\lim_{\epsilon \downarrow 0^+} \phi(\epsilon) = \phi(0) = 0.
\end{equation}
Invoking Monotonicity Lemma (see Lemma~\ref{Monotonicity Theorem}), there exists $\bar \epsilon>0$ such that
$\phi_{|_{[0,\bar\epsilon)}}$ is either constant $0$ or strictly decreasing.
Moreover, by (i), we may assume that $\phi$ is analytic on $(0,\bar\epsilon).$
If $\phi(\epsilon)=0$ for all $\epsilon\in[0,\bar\epsilon),$ then letting $c=0$ in \eqref{Lemma3rel1}, the desired result follows.
Otherwise, applying Growth Dichotomy Lemma (see Lemma~\ref{Growth Dichotomy Lemma}) (reducing $\bar\epsilon$ if necessary), we see that there exist constants $c \neq 0$ and $q \in \mathbb{Q}$ such that $\phi(\epsilon) = c\epsilon^q + o(\epsilon^q )$ as $\epsilon \to 0^+.$ Since $\phi$ is strictly decreasing,
\begin{equation*}
0 \ > \ \phi'(\epsilon) \ = \ c q \epsilon^{q-1} + o(\epsilon^{q-1} ),
\end{equation*}
and so $cq < 0.$ Finally, we deduce easily from \eqref{Lemma3rel2} that $c < 0$ and $q > 0.$
\end{proof}

\subsection{Solving the $\epsilon$-approximation via sequential SDP relaxations}

For simplicity, we write $z:=(x,y).$ Let $\mu$ be a finite Borel probability measure uniformly distributed on $\mathrm{\Omega} \subset \mathbb{R}^n\times\mathbb{R}^m.$
We will assume that $\mathrm{\Omega}$ is a simple compact set (e.g., a simplex, a box or an ellipsoid) so that the moments
\begin{equation*}
\gamma_{\alpha} := \int_{\mathrm{\Omega}} z^{\alpha} d \mu(z), \quad \alpha \in \mathbb{N}^n \times \mathbb{N}^m,
\end{equation*}
can be computed easily. For instance, in the sequel we will assume that
\begin{equation*}
\mathrm{\Omega} := \{z \in \mathbb{R}^n \times \mathbb{R}^m : h_j(z) \ge 0, \ j = s + 1, \ldots, s+n+m \}
\end{equation*}
with  $h_j \in \mathbb{R}[x, y]$ being the polynomial
\begin{equation*}
h_j(x, y) :=
\begin{cases}
M-x_j^2  & \textrm{ for } j = s + 1, \ldots, s + n,\\
M-y_j^2 & \textrm{ for } j = s + n + 1, \ldots, s+n+m,
\end{cases}
\end{equation*}
where $M>0$ is chosen so that $\mathrm{\Omega} \supset B.$

The following assumption, which is commonly used in polynomial optimization (see \cite{Lasserre2009, Ha2017} and references therein), plays an important key role for our results.

\begin{enumerate}
\item[(H2)] {\em The system $\left\{h_1,\ldots,h_s\right\}\subset \mathbb{R}[x,y]$ satisfies the Archimedean property.}
\end{enumerate}

\begin{remark}{\rm
(i) The assumption ${\rm (H2)}$ implies that the set $B$ is compact but the inverse is not necessarily true. However, if $B$ is compact and one knows a bound $R$ for $\|(x, y)\|$ on $B,$ then it suffices to add the ``redundant'' quadratic constraint $h_0(x, y) := R^2 - \|(x, y)\|^2 \ge 0$ to the definition of $B,$ and ${\rm (H2)}$ holds.

(ii) When the set $B$ is compact and the assumption ${\rm (H2)}$ does not hold, there is still a representation of polynomials, strictly positive on $B$ (see Corollary 3 in Schm\"udgen \cite{Schmudgen1991}). But, instead of being ``linear'' as in \eqref{Eqn7} below, there are product terms of the form $h_{j_1} \cdots h_{j_l}$ times a sum of squares of polynomials, with $j_1, \ldots, j_l \in \{1, \ldots, s + n + m\}.$ However,
the size of this semidefinite programming will grow exponentially with the number of constraints $s$ and the number of variables $n + m.$
}\end{remark}

For each $k \in \mathbb{N},$ with $k \ge  \max\{\lceil \frac{\deg \varphi}{2}\rceil, \lceil \frac{\deg h_j}{2}\rceil\},$ (where the notation $\lceil a\rceil$ stands for the smallest integer that is greater than or equal to $a,$) consider the following optimization problem
\begin{eqnarray} \label{Eqn7}
&\sup\limits_{p, (\sigma_j)} & \int_{\mathrm{\Omega}} p d \mu(z) =
\left( \sum_{\alpha \in \mathbb{N}_{2k}^{n + m}} p_\alpha \gamma_\alpha \right) \nonumber\\
& {\rm s.t.} & \varphi(x,y,v) - p(x, y) \ = \ \sigma_0 + \sum_{j = 1}^{s} \sigma_j h_j(x, v) + \sum_{j = s +n+1}^{s+n+m} \sigma_j h_j(x, y), \\
& & p := \sum_{\alpha \in \mathbb{N}_{2k}^{n + m}} p_\alpha z^\alpha \in \mathbb{R}[x, y], \ \sigma_j \in \mathrm{\Sigma}^2[x, y, v], \nonumber  \\
& & \deg \sigma_0 \le 2k, \ \deg (\sigma_j h_j) \le 2k, j \in \{1, \ldots, s, s + n + 1, \ldots, s+n+m\}. \nonumber
\end{eqnarray}
It is not hard to see that this is a semidefinite programming (see, for example, \cite{Ha2017,Lasserre2009}).
We also should mention that semidefinite programs can be solved (approximatively) in polynomial time, using the interior point methods.
For more details on semidefinite programming, the interested reader is referred to Vandenberghe and Boyd \cite{Vandenberghe1996}.

Lemma~\ref{TheoremLasserre} below shows that any optimal solution of Problem~\eqref{Eqn7} permits to approximate $J$ in a strong sense.
Note that we do not include the polynomials $h_j(x,y)$ for $j=s+1,\ldots,s+n$ in \eqref{Eqn7} as is usual.
It turns out that we do not need these polynomials to obtain the results in this paper, as we shall see (compare \cite{Jeyakumar2016,Lasserre2012}).

\begin{lemma} \label{TheoremLasserre}
Suppose that the assumptions ${\rm (H1)}$  and ${\rm (H2)}$ hold.
Let $\rho_k$ be the optimal value of the semidefinite program~\eqref{Eqn7} and let $((p_\alpha), (\sigma_j))$ be an optimal or a $\frac{1}{k}$-solution of \eqref{Eqn7} $($i.e.$,$ such that $\sum_{\alpha} p_\alpha \gamma_{\alpha} \ge \rho_k - \frac{1}{k}).$ Let $J_k \in \mathbb{R}[z]$ be the polynomial $z := (x, y) \mapsto J_k(z) := \sum_{\alpha} p_\alpha z^\alpha.$
Then $J_k(z) \le J(z)$ for all $z\in \mathrm{\Omega}$ and
\begin{equation*}
\lim_{k\to\infty}\int_{\mathrm{\Omega}} |J_k(z) - J(z)| d \mu(z) = 0,
\end{equation*}
that is$,$ $J_k \to J$ for the $L_1(\mathrm{\Omega}, \mu)$-norm\footnote{$L_1(\mathrm{\Omega}, \mu)$ is the Banach space of $\mu$-integrable functions on $\mathrm{\Omega},$ with norm $\|f\| = \int_{\mathrm{\Omega}} |f| d \mu.$}.
\end{lemma}

\begin{proof}
The assumption ${\rm (H1)}$ implies that for any $(x, y, v) \in \mathbb{R}^n \times \mathbb{R}^m \times \mathbb{R}^m$ we have
$(x, y) \in \Omega$ and $v \in B(x)$ if and only if
\begin{eqnarray*}
h_j(x, v) \ge 0 \textrm{ for } j = 1, \ldots, s, \textrm{ and } h_j(x, y) \ge 0 \textrm{ for } j = s + n + 1, \ldots, s + n + m.
\end{eqnarray*}

On the other hand, the assumption ${\rm (H2)}$ gives the existence of $R>0$ and $\sigma_j\in\mathrm{\Sigma}[x,v],$ $j=0,1,\ldots,s,$ such that
\begin{equation*}
R-\|(x,v)\|^2=\sigma_0+\sum_{j=1}^s\sigma_jh_j(x,v).
\end{equation*}
Letting $R' := R + m \cdot M,$ we get
\begin{eqnarray*}
R' - \|(x,y,v)\|^2
&=& \sigma_0+\sum_{j=1}^s\sigma_jh_j(x,v)+\sum_{j=s+n+1}^{s+n+m} (M - y_j^2) \\
&=&  \sigma_0+\sum_{j=1}^s\sigma_jh_j(x,v)+\sum_{j=s+n+1}^{s+n+m}1\cdot h_j(x,y),
\end{eqnarray*}
which implies that the system $\{h_1,\ldots,h_s,h_{s+n+1},\ldots,h_{s+n+m}\}\subset \mathbb{R}[x,y,v]$ satisfies the Archimedean property.
Now, applying \cite[Theorem~3.5]{Lasserre2010}, the desired result follows.
\end{proof}

We now introduce a scheme to solve the $\epsilon$-approximation problem $({\rm P}_{\epsilon})$ for arbitrary $\epsilon>0$, using sequences of semidefinite programming relaxations.

\begin{algorithm}\label{algorithm1}
\noindent
\begin{enumerate}
  \item[{\rm Step 0:}] Fix  $\epsilon>  0$. Set $k=1$.
  \item[{\rm Step 1:}] Solve the semidefinite program \eqref{Eqn7} and obtain a $\frac{1}{k}$-solution $(p, \sigma_j)$ of Problem~\eqref{Eqn7}.
Define $J_{k}(z) := \sum_{\alpha \in \mathbb{N}_{2k}^{n + m}} p_{\alpha} z^{\alpha},$ $z:=(x,y)\in\mathbb{R}^n\times\mathbb{R}^m.$
  \item[{\rm Step 2:}] Consider the following basic closed semi-algebraic set
\begin{eqnarray*}
S_k \ := \ \{(x, y) \in \mathrm{\Omega}
&:& g_i(x, y) \ge-\epsilon, \ i=1,\ldots, r,  \ h_j(x, y) \ge-\epsilon, \ j=1, \ldots, s, \\
& & J_{k}(x, y) \ge -\epsilon\}.
\end{eqnarray*}

If $S_k = \emptyset$, then let $k=k+1$ and return to Step 1. Otherwise$,$ go to Step 3.

  \item[{\rm Step 3:}]Solve the following polynomial optimization problem
\begin{eqnarray}
&\min\limits_{(x,y) \in \mathrm{\Omega}}& f(x,y)  \notag\\
& {\rm s.t.} & g_i(x, y) \ge-\epsilon, \ i= 1,\ldots, r, \ h_j(x, y) \ge-\epsilon, \ j = 1, \ldots, s, \eqname{P$_{\epsilon}^k$} \\
& & J_{k}(x, y) \ge -\epsilon.\notag
\end{eqnarray}
  \item[{\rm Step 4:}] Let $v_{\epsilon}^k := \min_{1 \le i \le k} {\rm val}({\rm P}_{\epsilon}^i)$.  Update $k=k+1$. Go back to Step 1.
\end{enumerate}
\end{algorithm}

\begin{remark}{\rm
We would like to note that the feasibility problem of the semialgebraic set $S_k$ in Step~2 can be tested by an SDP hierarchy via the Positivstellnsantz; this was explained in \cite{Parrilo2003} and was implemented in the matlab toolbox SOSTOOLS. As explained before, Step 3 can also be accomplished by solving an
sequence of SDPs; for more details, we refer the reader to \cite{Lasserre2001, Lasserre2009}.
}\end{remark}

Next, we justify that Algorithm~\ref{algorithm1} is a legitimate procedure.

\begin{lemma}\label{Lemma4}
Suppose that assumptions ${\rm (H1)}$  and ${\rm (H2)}$ hold.
Let $\epsilon > 0$ be any fixed. Then the following two statements hold$:$
\begin{itemize}
\item[{\rm (i)}] There exists an integer $k_{\epsilon} > 0$ such that for all $k \ge k_\epsilon,$
\begin{equation*}
S_k \ne \emptyset \quad \textrm{ and } \quad v_\epsilon^{k} <  f_* + \epsilon.
\end{equation*}
\item[{\rm (ii)}] The limit $\lim_{k \to \infty} v_{\epsilon}^k$ exists and
\begin{equation*}
\mathrm{val}({\rm P}_{\epsilon}) \le \lim_{k \to \infty} v_{\epsilon}^k \le \displaystyle \lim_{\delta \to \epsilon^{-}}\mathrm{val}({\rm P}_{\delta}).
\end{equation*}
Further$,$ the equality $\mathrm{val}({\rm P}_{\epsilon}) = \lim_{k \to \infty} v_{\epsilon}^k$ holds for all $\epsilon \in (0, +\infty)$ except finitely many points.
\end{itemize}
\end{lemma}

\begin{proof}
(i) The assumptions that the feasible set of ${\rm (MPEC)}$ is nonempty and bounded imply that $f_\ast>-\infty,$ and so, there exists a point $z_{\epsilon} \in A \cap B$ such  that
\begin{equation*}
J(z_{\epsilon}) \ge 0 \quad \textrm{ and } \quad f(z_{\epsilon}) < f_* + \frac{\epsilon}{4}.
\end{equation*}
By a standard argument, it is easy to find a point $\tilde{z}_\epsilon\in {\rm int}\left(A_\epsilon\cap B_\epsilon\right)$ such that
\begin{equation*}
J(\tilde{z}_\epsilon)>-\frac{\epsilon}{2} \quad \textrm{ and } \quad f(\tilde{z}_\epsilon)<f_*+\frac{\epsilon}{2}.
\end{equation*}
Let
\begin{equation*}
U_1 := \{z \in {\rm int}(A_\epsilon \cap B_\epsilon) : J(z) > -\frac{\epsilon}{2}\}.
\end{equation*}
Then $\tilde{z}_\epsilon\in U_1.$
We will now show that $U_1$ is an open set.
Suppose on the contrary that there exists a sequence $\{\tilde{z}^k\}\subset {\rm int}(A_\epsilon\cap B_\epsilon)$ such that $\tilde{z}^k\to \tilde{z}\in U_1$ and $J(\tilde{z}^k)\le-\frac{\epsilon}{2}.$
This together with the lower semicontinuity of $J$ (see Lemma~\ref{Lemma1}) yields that
\begin{equation*}
-\frac{\epsilon}{2}<J(\tilde{z})\le\liminf_{k\to\infty}J(\tilde{z}^k)\le-\frac{\epsilon}{2},
\end{equation*}
which is a contradiction, and so, $U_1$ is a nonempty open set. Next, as $f$ is continuous, there exists $\rho > 0$ such that
\begin{equation*}
f(z) < f_* + \epsilon \quad \textrm{ whenever } \quad z \in U_2 := \{z \in {\rm int}(A_\epsilon \cap B_\epsilon) : \|z - \tilde{z}_{\epsilon}\| < \rho\}.
\end{equation*}
Observe that $\eta := \mu(U_1 \cap U_2) > 0$ because the set $U_1 \cap U_2$ is open and nonempty (because it contains the point $\tilde{z}_{\epsilon}).$
Note from Lemma~\ref{TheoremLasserre} that $J_k \to J$ for the $L_1(\mathrm{\Omega}, \mu)$-norm.
Hence $J_k$ converges to $J$ almost everywhere on $\mathrm{\Omega}.$
As $\mu(\mathrm{\Omega})<+\infty,$ the classical Egorov's theorem (see, for example, \cite[Theorem 2.5.3]{Ash1972}) implies that there is a subsequence $\{k_\ell\}_{\ell \in \mathbb{N}}$  such that $J_{k_\ell} \to J,$ $\mu$-almost uniformly on $\mathrm{\Omega}.$
Hence, there are some Borel set $\mathrm{\Sigma} \subset \mathrm{\Omega}$ and integer $\ell_\epsilon \in \mathbb{N}$ such that
\begin{equation*}
\mu(\mathrm{\Sigma}) < \frac{\eta}{2} \quad \textrm{ and  } \quad \sup_{z \in \mathrm{\Omega} \setminus \mathrm{\Sigma}} |J(z) - J_{k_\ell}(z)| < \frac{\epsilon}{2} \ \textrm{ for all } \ \ell \ge \ell_\epsilon.
\end{equation*}
In particular, as $\mu(\mathrm{\Sigma}) < \frac{\eta}{2} < \mu(U_1 \cap U_2),$ the set $(U_1 \cap U_2) \setminus \mathrm{\Sigma}$ is nonempty.
Therefore,
\begin{eqnarray*}
S_{k_\ell} \ne \emptyset, \quad J_{k_\ell}(z) > -\epsilon, \quad \textrm{ and } \quad f(z) < f_* + \epsilon,
\end{eqnarray*}
for all $ \ell \ge \ell_\epsilon$ and all $z \in (U_1 \cap U_2) \setminus \mathrm{\Sigma},$ which in turn implies $\mathrm{val}({\rm P}_\epsilon^{k_\ell}) < f_* + \epsilon,$ and consequently, $v_\epsilon^{k_\ell} <  f_* + \epsilon,$ the desired result.

(ii) Recall from Lemma~\ref{TheoremLasserre} that $J_k(z) \le J(z)$ for all $k \in \mathbb{N}$ and for all $z \in \mathrm{\Omega}.$
Then, by definition, ${\rm val}({\rm P}_{\epsilon}^k) \ge {\rm val}({\rm P}_{\epsilon})$ for all $k \in \mathbb{N},$ which in turn implies $v_{\epsilon}^k \ge {\rm val}({\rm P}_{\epsilon})$ for all $k \in \mathbb{N}.$ Note that $v_{\epsilon}^k$ is a non-increasing sequence which is bounded below, and so the limit $v_{\epsilon} := \lim_{k \to \infty}v_{\epsilon}^k$ exists. Moreover, we have
\begin{equation*}
v_{\epsilon}  \ge {\rm val}({\rm P}_{\epsilon}).
\end{equation*}

On the other hand, replacing $({\rm P})$ and $({\rm P}_{\epsilon})$ by $(P_{\epsilon})$ and $(P_{\epsilon - \delta}),$ respectively, in Item~(i), it not hard to see that for every $\delta \in (0,\epsilon),$ there exists an integer $k_\delta > 0$ such that for all $k \ge k_\delta,$
\begin{equation*}
v_{\epsilon}^k \le \mathrm{val}(P_{\epsilon - \delta}) + \delta.
\end{equation*}
Therefore,
\begin{equation*}
{\rm val}(P_{\epsilon}) \le v_{\epsilon} \le \displaystyle \lim_{\delta \to \epsilon^{-}}{\rm val}(P_{\delta}).
\end{equation*}

To see the last assertion in Item~(ii), we only need to notice from Lemma~\ref{Lemma3}(i) that $\epsilon \mapsto {\rm val}(P_{\epsilon})$ is continuous except finitely many points over $(0, +\infty)$.
\end{proof}

\begin{definition}{\rm
For $\epsilon,\delta\ge0,$ a point $(\bar x, \bar y) \in \mathbb{R}^n \times \mathbb{R}^m$ is called a {\em $\delta$-solution} of $({\rm P}_\epsilon^k)$ if $(\bar x, \bar y)$ is feasible for $({\rm P}_\epsilon^k)$ and $f(\bar x, \bar y)\le \mathrm{val}{\rm (P_\epsilon)} + \delta.$
}\end{definition}

We now establish the main theorem of this paper which is the convergence result of Algorithm~\ref{algorithm1}.
\begin{theorem} \label{Theorem1}
Under the assumptions of Lemma~\ref{Lemma4}$,$ suppose that the function $J\colon \mathrm{\Omega}\to\mathbb{R}$ is continuous.
Then there exists $\epsilon_0 > 0$ such that $\lim_{k \to \infty} v_{\epsilon}^k = \mathrm{val}({\rm P}_{\epsilon})$ for all $\epsilon \in (0,\epsilon_0).$
Moreover$,$ let $v_{\epsilon}^k := \min_{1 \le i \le k} \mathrm{val}({\rm P}_{\epsilon}^i)=\mathrm{val}({\rm P}_{\epsilon}^{i_k}),$ and let $(x^k,y^k)$ be a $\frac{1}{k}$-solution of Problem~$({\rm P}_{\epsilon}^{i_k})$.
Then$,$ $\{(x^k,y^k)\}$ is a bounded sequence and any cluster point $(\widehat{x}, \widehat{y})$ of $(x^k,y^k)$ is a global minimizer of Problem~$({\rm P}_{\epsilon})$ for all $\epsilon \in (0,\epsilon_0).$
\end{theorem}

\begin{proof}
By Lemma~\ref{Lemma4}(ii), there exists $\epsilon_0 > 0$ such that
\begin{equation*}
\lim_{k \to \infty}  v_{\epsilon}^k = \mathrm{val}({\rm P}_{\epsilon}) \quad \textrm{ for all } \quad \epsilon \in (0,\epsilon_0).
\end{equation*}
Now, fix any $\epsilon \in (0,\epsilon_0).$ Let $v_{\epsilon}^{k} := \min_{1 \le i \le k} {\rm val}({\rm P}_{\epsilon}^i)={\rm val}({\rm P}_{\epsilon}^{i_k}),$ and let $(x^k,y^k)$ be a $\frac{1}{k}$-solution of $({\rm P}^{i_k}_{\epsilon}).$
Then, $\{(x^k,y^k)\} \subseteq A_\epsilon\cap B_\epsilon,$ which is a compact set, and so, without loss of generality, we assume that there is a point $(\hat x,\hat y)\in A_\epsilon\cap B_\epsilon$ such that $(x^k,y^k)\to (\hat x,\hat y)$ as $k\to\infty.$
Note that $J \ge J_k$ on $\mathrm{\Omega}$ for all $k \in \mathbb{N}$ and $(x^k,y^k)$ is feasible for $({\rm P}_{\epsilon}^{i_k}).$
Then, for each $k \in \mathbb{N},$
\begin{equation*}
J(x^k, y^k) \ge J_{i_k}(x^k, y^k) \ge - \epsilon.
\end{equation*}
Passing to the limit and note from the assumption that $J$ is continuous, we have $J(\hat{x}, \hat{y}) \ge  -\epsilon.$
So, $(\hat{x}, \hat{y})$ is feasible for $({\rm P}_{\epsilon})$.
Finally, since $\lim_{k \to \infty}  v_{\epsilon}^k  = \textrm{val}({\rm P}_{\epsilon}),$ we have
\begin{equation*}
f(\hat{x},\hat{y})=\lim_{k \to \infty}f(x^k,y^k) \le  \lim_{k \to \infty}(v_{\epsilon}^{i_k} +\frac{1}{k})= {\rm val}({\rm P}_{\epsilon}),
\end{equation*}
where the inequality follows from the assumption that $(x^k,y^k)$ is a $\frac{1}{k}$-solution of $({\rm P}_{\epsilon}^{i_k}).$
Thus, $(\hat{x},\hat{y})$ is a global minimizer of $({\rm P}_{\epsilon})$.
\end{proof}

The following example illustrates how to solve Problem (MPEC) with our method.
\begin{example}{\rm
Let us consider the following mathematical programming with equilibrium constraints:
\begin{eqnarray}
&\displaystyle \min_{(x,y) \in \mathbb{R}^2}& x+y \notag\\
& \mbox{s.t.} & (x,y)\in A:=\{(x,y)\in\mathbb{R}^2:-x^2\left((xy-1)^2+y^4\right)\ge0\},\notag \\
&&y\in B(x):=\{y\in\mathbb{R}:1-x^2\ge0,\ 1-y^2\ge0\}, \eqname{P$_1$} \\
& & \varphi(x, y, v) := \frac{xv^2}{2}-\frac{v^3}{3}-\left(\frac{xy^2}{2}-\frac{y^3}{3}\right) \ge 0, \forall v \in B(x).\notag
\end{eqnarray}
A simple computation shows that $B(x)=[-1,1]$ and
\begin{eqnarray*}
J(x,y)&=&\min_{v \in [-1,1]}\varphi(x, y, v)\\
&=&\left\{
   \begin{array}{cl}
     -\frac{xy^2}{2}+\frac{y^3}{3} & \quad \textrm{if } \ (x,y)\in \left[\frac{2}{3},1\right]\times\left[-1,1\right], \\
     \frac{x}{2}-\frac{1}{3}-\frac{xy^2}{2}+\frac{y^3}{3} & \quad \textrm{if } \ (x,y)\in \left[-1,\frac{2}{3}\right)\times\left[-1,1\right].
   \end{array}
 \right.
\end{eqnarray*}
Moreover, we can easily verify that the feasible set of Problem $({\rm P}_1)$ is $\{(0,1)\},$ and so, the optimal solution of Problem $({\rm P}_1)$ is $(0,1)$ and the optimal value is $1.$
Letting $\mathrm{\Omega}:=[-1,1]\times [-1,1],$ all the assumptions in Theorem~\ref{Theorem1} are satisfied.

For $k=3,$ using GloptiPoly 3 \cite{Henrion2009}, we obtain a degree $2k(=6)$ polynomial approximation of $J(x,y),$ that is,
\begin{align*}
J_3(x,y)=&-0.3338+0.5011x+0.0098x^2-0.0032x^3-0.5xy^2+0.3333y^3\\
&-0.0696x^4-0.1013x^5-0.0432x^6.
\end{align*}

Setting $\epsilon=0.0005$ and solving the $\epsilon$-approximation of Problem $({\rm P}_1)$
\begin{eqnarray*}
&\displaystyle \min_{(x,y) \in \mathrm{\Omega}}& x+y \\
& \mbox{s.t.} & -x^2\left((xy-1)^2+y^4\right)\ge-\epsilon, \ 1-y^2\ge-\epsilon, \\
& & J_3(x,y)\ge-\epsilon,
\end{eqnarray*}
with GloptiPoly 3, we obtain the point $(-0.0157,1.0000)$ with its associated function value $0.9843,$  which are a good approximation of the optimal solution and optimal value of Problem $({\rm P}_1),$ respectively.
}\end{example}

\begin{remark}\label{rmk2}{\rm
It is worth emphasizing here that our method has an advantage over the ones in \cite{Jeyakumar2016,Lasserre2012} in the sense that
the assumptions concerning the interior of the feasible sets in these papers are no longer necessary.
Moreover, it is not hard to check that we can solve the problems considered in these papers directly by using our method.
We leave the details to the reader.
}\end{remark}

We close this section with two illustrated examples for Remark~\ref{rmk2}.

\begin{example}\label{ex3}{\rm
Let us consider the following bilevel optimization problem (cf. \cite[Example~4.8]{Jeyakumar2016}):
\begin{eqnarray}
&\displaystyle \min_{(x,y) \in \mathbb{R}^2}& x+y \notag\\
& \mbox{s.t.} & -x^2\left(x^2+(y-1)^2-1\right)\ge0,\eqname{P$_2$} \\
&& y\in Y(x):=\argmin_{v\in\mathbb{R}}\left\{\frac{xv^2}{8}-\frac{v^3}{24}:4-v^2\ge0\right\}.\notag
\end{eqnarray}
Let $K:=\left\{(x,y)\in\mathbb{R}^2 : -x^2\left(x^2+(y-1)^2-1\right)\right\}\ge0$ and $F:=\{y\in\mathbb{R}:4-y^2\ge0\}=[-2,2].$
Then, a simple calculation shows that
\begin{equation*}
K\cap(\mathbb{R}\times F)=\{0\}\times[-2,0] \cup \{(x,y)\in\mathbb{R}^2:x^2+(y-1)^2\le1\},
\end{equation*}
and so, ${\rm cl}\left({\rm int}\left(K\cap\left(\mathbb{R}\times F\right)\right)\right)\neq K\cap(\mathbb{R}\times F).$
This says that the assumption in \cite{Jeyakumar2016} does not hold for this problem, and so, it can not be solved this problem by the method in \cite{Jeyakumar2016}.

On the other hand, it is easy to check that
\begin{eqnarray*}
\tilde{J}(x)
=\min_{v \in [-2,2]}\left\{\frac{xv^2}{8}-\frac{v^3}{24}\right\}
=\left\{\begin{array}{cl}
     0 & \quad \textrm{if } \ x\in[\frac{2}{3},1], \\
     \frac{x}{2}-\frac{1}{3} & \quad \textrm{if } \ x\in [-1,\frac{2}{3}).
\end{array}\right.
\end{eqnarray*}
Moreover, we can see that the solution set of the lower-level problem $Y(x)$ is formulated as
\begin{eqnarray*}
Y(x)
=\left\{\begin{array}{cl}
     \{0\} & \quad \textrm{if } \ x\in(\frac{2}{3},1], \\
     \{0,2\} & \quad \textrm{if } \ x=\frac{2}{3}, \\
     \{2\} & \quad \textrm{if } \ x\in[-1,\frac{2}{3}). \\
     \end{array}\right.
\end{eqnarray*}
This yields that the feasible set of Problem $({\rm P}_2)$ is $\{(0,2)\},$ and so, the optimal solution of Problem $({\rm P}_2)$ is $(0,2)$ and the optimal value is $2.$

Let $\mathrm{\Omega}:=[-1,1].$
Then, for $k=3,$ using GloptiPoly 3 \cite{Henrion2009}, we obtain a degree $2k(=6)$ polynomial approximation of $\tilde{J}(x),$ that is,
\begin{eqnarray*}
\tilde{J}_3(x)&\approx&-0.3338+0.5011x+0.0098x^2-0.0032x^3-0.0696x^4\\
&&-0.1012x^5-0.0432x^6.
\end{eqnarray*}
Setting $\epsilon=0.001$ and solving the $\epsilon$-approximation of Problem $({\rm P}_2)$
\begin{eqnarray*}
&\displaystyle \min_{(x,y) \in \mathbb{R}^2}& x+y \\
& \mbox{s.t.} & -x^2\left(x^2+(y-1)^2-1\right)\ge-\epsilon, \\
&& 4-y^2\ge-\epsilon, \\
& & J_3(x)-\left(\frac{xy^2}{8}-\frac{y^3}{24}\right)\ge-\epsilon,
\end{eqnarray*}
with GloptiPoly 3, we obtain the point $(-0.0039,1.9992)$ with its associated function value $1.9953,$  which are a good approximation of the optimal solution and optimal value of Problem $({\rm P}_2),$ respectively.
}\end{example}

\begin{example}{\rm
Consider the following semi-infinite optimization problem:
\begin{eqnarray}
&\displaystyle \min_{(x_1,x_2)\in \mathbb{R}^2}& x_2\notag \\
& \mbox{s.t.} & -x_1^2\left(x_1^2+(x_2-1)^2-1\right)\ge0, \ 4-x_2^2\ge0, \eqname{P$_3$}\\
&&-2x_1^2v^2+v^4-x_1^2+x_2\ge0, \ \forall v\in[-1,1].\notag
\end{eqnarray}
Let $X:=\left\{(x_1,x_2)\in\mathbb{R}^2 : -x_1^2\left(x_1^2+(x_2-1)^2-1\right)\ge0, \ 4-x_2^2\ge0\right\}\ge0.$
Then, as we have already seen from Example~\ref{ex3},
\begin{equation*}
X=\{0\}\times[-2,0] \cup \{(x_1,x_2)\in\mathbb{R}^2:x_1^2+(x_2-1)^2\le1\},
\end{equation*}
which is not the closure of an open set.
It means that the assumption in \cite[Theorem~3.4]{Lasserre2012} does not satisfy for this problem, and so, we may not solve this problem with the method, which is described in \cite{Lasserre2012}.

A simple calculation gives us the function $\Phi\colon\mathbb{R}^2\to\mathbb{R},$
\begin{eqnarray*}
(x_1,x_2)\mapsto\Phi(x_1,x_2)&=&\min_{v \in [-1,1]}\left\{-2x_1^2v^2+v^4-x_1^2+x_2\right\}\\
&=&x_2-x_1^2-x_1^4.
\end{eqnarray*}
It is worth noting that the best known optimal solution of Problem $({\rm P}_3)$ is $(0,0)$ and the best known optimal value is $0.$

Let $\mathrm{\Omega}:=[-1,1]\times[-2,2].$
Then, for $k=2,$ using GloptiPoly 3 \cite{Henrion2009}, we obtain a degree $2k(=4)$ polynomial approximation of $J(x_1,x_2),$ that is,
\begin{eqnarray*}
\Phi_2(x_1,x_2)\approx x_2-x_1^2-x_1^4,
\end{eqnarray*}
which is a good approximation of $\Phi(x_1,x_2).$

Setting $\epsilon=0.0001$ and solving the $\epsilon$-approximation of Problem $({\rm P}_3)$
\begin{eqnarray*}
&\displaystyle \min_{(x_1,x_2) \in \mathbb{R}^2}& x_2 \\
& \mbox{s.t.} &  -x_1^2\left(x_1^2+(x_2-1)^2-1\right)\ge-\epsilon, \ 4-x_2^2\ge-\epsilon, \\
&&\Phi_2(x_1,x_2)\ge-\epsilon,
\end{eqnarray*}
with GloptiPoly 3, we obtain the point $(\bar{x}_1,\bar{x}_2)=(-0.0000,-0.0001)$ with its associated function value $-0.0001,$  which are a good approximation of the optimal solution and optimal value of Problem $({\rm P}_3),$ respectively.
}\end{example}

\section{Conclusions}\label{sect:4}

This paper studies how to solve a polynomial mathematical program with equilibrium constraints. We have proposed a method for finding its global minimizers and global minimum using a sequence of semidefinite programming relaxations and have proved the convergence result for the method. It was different from the papers \cite{Jeyakumar2016,Lasserre2012} that we do not make technical assumptions concerning the interior of the feasible sets. As a byproduct, bilevel polynomial optimization problems and semi-infinite optimization problems, which can be regarded as special cases of Problem (MPEC), were also solvable directly based on our approach.

\subsection*{Acknowledgments}
The authors would like to express their sincere thanks to Guoyin~Li for his warm help for the paper.
The final version of this paper was completed while the third author was visiting at the Vietnam Institute for Advanced Study in Mathematics (VIASM) from January 1 to 31 March, 2019. He would like to thank the Institute for hospitality and support.

\end{document}